\documentclass[12pt]{amsart}

\usepackage{amsfonts,latexsym,rawfonts,amsmath,amssymb,amsthm,graphicx}
\newtheorem{theorem}{Theorem}[section]

\newtheorem{remark}[theorem]{Remark}
\newtheorem{lemma}[theorem]{Lemma}

\def\bee{\begin{eqnarray}}
\def\beee{\begin{eqnarray*}}
\def\eee{\end{eqnarray}}
\def\eeee{\end{eqnarray*}}

\def\ba{\begin{array}}
\def\ea{\end{array}}

\def\R{\mathbb R}

\numberwithin{equation}{section}

\begin{document}

\title[Lorentzian harmonic maps]
{A global weak solution to the Lorentzian harmonic map flow}

\author{Xiaoli Han, Lei Liu, Liang Zhao}

\address{Xiaoli Han, Department of Mathematical Sciences, Tsinghua University \\ Beijing 100084, China}
\email{xlhan@math.tsinghua.edu.cn}

\address{Max Planck Institute for Mathematics in the Sciences\\ Inselstrasse 22\\ 04103 Leipzig, Germany}
\email{leiliu@mis.mpg.de or llei1988@mail.ustc.edu.cn}

\address{Liang Zhao, School of Mathematical Sciences,
Beijing Normal University\\
Laboratory of Mathematics and
Complex Systems, Ministry of Education\\ Beijing 100875, China}
\email{liangzhao@bnu.edu.cn}

\thanks {The research is supported by NSFC No.11471014, No.11471299 and the Fundamental Research Funds for the Central Universities}

\subjclass[2010]{  }

\date{\today}

\maketitle

\begin{abstract}
We investigate a parabolic-elliptic system which is related to a harmonic map from a compact Riemann surface with a smooth boundary into a Lorentzian manifold with a warped product metric. We prove that there exists a unique global weak solution for this system which is regular except for at most finitely many singular points.
\end{abstract}

\section{introduction}

To describe our problem, we present some notations first. Let $(M, h)$ be a compact Riemann surface with a smooth boundary $\partial M$ and $N\times{\mathbb{R}}$ be a Lorentzian manifold equipped with a warped product metric of the following form
$$g=g_N-\beta d^2\theta,$$
where $({\mathbb{R}}, d\theta^2)$ is the $1$-dimensional Euclidean space, $(N, g_N)$ is an $n$-dimensional compact Riemannian manifold which, by Nash's embedding theorem, is embedded isometrically into some Euclidean space ${\mathbb{R}}^K$ and $\beta$ is
a positive $C^{\infty}$ function on $N$. Since $N$ is compact, there exist two positive constants $\lambda$ and $\Lambda$ such that
$$
0<\lambda\leq \beta(y)\leq \Lambda, \forall y\in N.
$$
For more details on such kind of manifolds, we refer to \cite{KSHM, ONeill}.

We consider the following Lagrangian
\begin{equation}\label{lag}
E_g(u, v)= \frac{1}{2}\int_{M} \left\{|\nabla u|^2- \beta(u)|\nabla v|^2 \right\} dv_h,
\end{equation}
which is called the Lorentzian energy of a map $(u,v)$ from $M$ to $N\times{\mathbb{R}}$. It is easy to see that $E_g(\cdot,\cdot)$ is conformally invariant in dimension two. A critical point $(u,v)$ of the Lagrangian \eqref{lag} is called a Lorentzian harmonic map from $M$ into the Lorentzian manifold  $(N\times{\mathbb{R}},g)$.

Via direct calculation, one can derive the  Euler-Lagrange equations for \eqref{lag},
\begin{equation}\label{Eleq}
 \left\{
  \begin{array}{cc}
    -\Delta u    = A(u)(\nabla u, \nabla u)-B^{\top}(u)|\nabla v|^2,\ &in\ M, \\
     -div\{\beta(u)\nabla v\} = 0 ,\ &in \ M.\\
  \end{array}
\right.
\end{equation}
where $A$ is the second fundamental form of $N$ in ${\mathbb{R}}^K$, $B(u):=(B^1, B^2, \cdots, B^K)$ with
$$
B^j:=-\frac{1}{2}\frac{\partial\beta(u)}{\partial y^j}
$$
and $B^{\top}$ is the tangential part of $B$ along the map $u$. For  details, see \cite{Zhu}.
We denote
$$
E(u;\Omega)=\frac{1}{2}\int_{\Omega} |\nabla u|^2 dx,
$$
and
$$
E(v;\Omega)=\frac{1}{2}\int_{\Omega} |\nabla v|^2 dx,
$$
where $\Omega$ is a bounded domain in $M$. For brevity, we will omit $\Omega$ if the domain is clear from the context.

The regularity theory of Lorentzian harmonic maps was studied in \cite{Isobe-2,Zhu} for dimension two and in \cite{Isobe-1} for some higher dimensional minimal type solutions. For partial regularity of stationary Lorentzian harmonic maps, one can refer to \cite{Li-Liu}. When the target manifold is a Lorentzian manifold, the existence of geodesics was proved in \cite{bencifortunatogiannoni}. In \cite{greco1,greco2}, Greco constructed a smooth harmonic map via some developed variational methods.

For the Riemannian case, Al'ber \cite{alber1,alber2},  Eells-Sampson \cite{EellsSampson} and Hamilton \cite{hamilton} proved an existence result using the harmonic map heat flow when the sectional curvature of the target manifold is non-positive. The heat flow associated with the Euler-lagrange equations can be interpreted as a gradient flow for the energy functional. In the Lorentzian case, because of the sign convention, although the Lorentzian energy $E_g$ still decreases along the corresponding heat flow, we can even not do blow-up analysis under an energy condition such as $E_g\leq \Lambda$ for some constant $\Lambda>0$. In order to construct nontrivial
Lorentzian harmonic maps in the homotopic class of the initial map and overcome the above obstacles, Han-Jost-Liu-Zhao \cite{HJLZ} introduced the following parabolic-elliptic system
\begin{align}\label{heateq}
\begin{cases}
\partial_t u=\Delta u+ A(u)(\nabla u,\nabla u)-B^\top(u)|\nabla v|^2,\ &in\ M\times [0,T), \\
-div (\beta(u)\nabla v)=0, \ &in\ M\times [0,T).
\end{cases}
\end{align}
They proved the local existence of a regular solution to \eqref{heateq} with the boundary-initial data
\begin{align}\label{boundary-data}
\begin{cases}
u(x,t)=\phi(x), \  &on\  \partial M\times\{t\geq 0\},\\
u(x,0)=\phi_0(x), \  &in\  M,\\
v(x,t)=\psi(x),\  &on \  \partial M\times\{t\geq 0\},\\
\phi_0(x)=\phi(x),\  &on\  \partial M.
\end{cases}
\end{align}

\begin{theorem}[Theorem 3.3, \cite{HJLZ}]\label{thm:shortime-existence}
For any $$\phi_0\in C^{2+\alpha}(M,N),\ \phi\in C^{2+\alpha}(\partial M,N),\ \psi\in C^{2+\alpha}(\partial M, \R)$$ where $0<\alpha<1$, the problem \eqref{heateq} and \eqref{boundary-data}  admits a unique solution $$u\in \cap_{0<s<T} C^{2+\alpha,1+\alpha/2}(M\times [0,s]),$$ and
$$v,\nabla v\in \cap_{0<s<T} C^{\alpha,\alpha/2}(M\times [0,s]),\ v\in L^\infty([0,T);C^{2+\alpha}(M)),$$ for some time $T>0$. Here, the maximum existence time $T$ is characterized by the condition that
$$\limsup_{x\in M,t\to T}E(u;B^M_r(x))>\epsilon_0\mbox{ for any } r>0,$$
where $\epsilon_0$ is a constant depending only on $M,N,\phi,\psi$ and $\phi_0$ and $B^M_r(x)$ is a geodesic ball in $M$. Moreover, the set
\begin{align}\label{def:singular-set}
S(u,T):=\{x\in M|\limsup_{t\to T}E(u;B^M_r(x))>\epsilon_0\mbox{ for any } r>0\}
\end{align}
is finite and a point in it is called a singularity at the singular time $T$.
\end{theorem}

By using the blow-up analysis, they \cite{HJLZ} got a global existence result to the problem \eqref{heateq} and \eqref{boundary-data} by assuming either some geometric conditions on the target manifold or small energy of the boundary-initial data. The result implies the existence of Lorentzian harmonic maps in a given homotopy class. For the further blow-up behavior, including the energy identities and no neck properties, we refer to \cite{hanzhaozhu} for Lorentzian harmonic maps and \cite{HJLZ-02} for approximate Lorentzian harmonic maps and the corresponding heat flow.

Inspired by works of Struwe for harmonic map flow in \cite{St}, our main purpose in this paper is to study the existence of a global weak solution of (\ref{heateq}) with the boundary-initial data \eqref{boundary-data} in the function space
\begin{eqnarray*}
V(M_s^t):&=&\{(u, v): M \times [s, t]\rightarrow N\times {\mathbb{R}}| \sup_{s\leq\tau\leq t}\|\nabla u\|_{L^2(M)}+\sup_{s\leq\tau\leq t}\|\nabla v\|_{L^4(M)}\\
&&+\int_s^t\int_M (|\partial_t u|^2+|\nabla^2 u|^2)dM dt<\infty\}.
\end{eqnarray*}
Here and in the sequel, we use the notations $M_s^t=M\times [s, t]$ and $M^T=M\times [0, T]$. $  C^{2+\alpha, 1+\frac{\alpha}{2}}(M^T)$ denote the usual H\"{o}lder spaces and $W^{k,p}(M,N)$ denote the Sobolev space $$W^{k,p}(M,N)=\{u\in W^{k,p}(M,\R^K),\ u(x)\in N\ for \ a.e.\ x\in M\}.$$

Precisely, our main result is

\begin{theorem}\label{global}
For any $\phi_0\in W^{1,2}(M, N)$, $\phi\in C^{2+\alpha}(\partial M, N)$ and $\psi\in C^{2+\alpha}(\partial M, {\mathbb{R}})$, there exists a unique global weak solution $(u,v)$ of (\ref{heateq}) with the boundary-initial data \eqref{boundary-data} on $M\times [0,\infty)$.

Moreover, there exist at most finitely many singular points $\{(x^l,T_k)\}_{l=1}^{l_k}$ which is characterized by
\begin{equation}\label{singularity}
\limsup_{t\rightarrow T_k}E(u(t); B_r^M(x^l))>\overline{\epsilon}\ \ \text{for all}\ \ r>0.
\end{equation}
Here $\overline{\epsilon}$ is the constant defined in \eqref{equat:02}, $B_R^M(x)$ is the geodesic ball in $M$, $1\leq l\leq l_k$, $1\leq k\leq K$ and $\sum_{k=1}^Kl_k\leq L$ for some non-negative integers $K$ and $L$ depending on $M$, $N\times {\mathbb{R}}$, $\Lambda$, $\phi,\ \psi$ and $E(\phi_0)$. The solution is regular outside the singular set, which means that
$u\in C_{loc}^{2+\alpha,1+\alpha/2}(M\times (0,\infty))$ except for singular time $ \{T_k\}_{k=1}^K$ and $\ u,\nabla u\in C_{loc}^{\alpha,\alpha/2}(M\times (0,\infty))$ outside singular points $ \cup_{k=1}^K\{(x^l,T_k)\}_{l=1}^{l_k}$. Moreover, $\ v,\nabla v\in C_{loc}^{\alpha,\alpha/2}(M\times (0,\infty))$, $v\in L^\infty\left((0,\infty);C^{2+\alpha}(M)\right)$ except for singular time $ \{T_k\}_{k=1}^K$ and $v\in L^\infty\left((0,\infty);C^{2+\alpha}_{loc}(M)\right)$ outside singular points $ \cup_{k=1}^K\{(x^l,T_k)\}_{l=1}^{l_k}$

Finally, there exist a nontrivial Lorentzian harmonic map $(u_\infty,v_\infty)\in C^{2+\alpha}(M,N\times\R)$ with the boundary data $(u_\infty,v_\infty)|_{\partial M}=(\phi,\psi)$ and a time sequence $t_i\rightarrow \infty$, such that $\{(u,v)(\cdot, t_i)\}$ converges to $(u_\infty,v_\infty)$ weakly in $W^{1,2}(M)$ as $i\rightarrow \infty$ and strongly in $C^1(M\setminus \mathcal{S}_\infty)$ where $\mathcal{S}_\infty$ is a finite points set defined by \eqref{def:01}.
\end{theorem}

In this paper, we just focus on the two dimension case. For the higher dimensions, they need more techniques including the new monotonicity formula, regularity theorem and so on. We leave it to the sequel to this paper.

The rest of this paper is organized as follows. In Section 2, we firstly recall some results which will be used in this paper. Secondly, we derive some basic lemmas, including a priori $W^{2,2}$-estimates and uniqueness for the weak solution. Our main Theorem \ref{global} will be proved in Section 3.

\section{preliminary results}

In this section,we will establish some basic lemmas for the Lorentzian harmonic map flow.

By the standard elliptic theory, for $\phi\in C^{2+\alpha}(\partial M)$, there exists a unique solution $u\in C^{2+\alpha}(M)$ of the equation
\begin{align}\label{equat:01}
\begin{cases}
\Delta u=0\ &in\ M,\\
u(x)=\phi(x)\ &on \ \partial M,
\end{cases}
\end{align}
satisfying $$\|u\|_{C^{2+\alpha}(M)}\leq C(\alpha,M)\|\phi\|_{C^{2+\alpha}(\partial M)}.$$
We call this solution $u$ an extension of $\phi$. Similarly, there also exists a harmonic function $v$ with the boundary data $\psi$, which is called an extension of $\psi$. For simplicity, we still denote $(u,v)$ by $(\phi,\psi)\in C^{2+\alpha}(M)$ and in the following, we use the extension when  needed.

Firstly, let us recall some results which will be used in this paper.

\begin{lemma}[Theorem 2.2 and Remark 2.1, P. 62, P. 63 in \cite{lady} or Lemma 4.1 in \cite{chenlevine}]\label{lem:01}
For any smooth bounded domain $\Omega\subset \mathbb{R}^2$ and any function $u\in W^{1,2}(\Omega)$, there exists a constant $C>0$ depending on the shape of $\Omega$ such that
\begin{equation}
\int_\Omega|u|^4dx\leq C \int_\Omega|u|^2dx\big(\int_\Omega|\nabla u|^2dx+\frac{1}{|\Omega|}\int_\Omega|u|^2dx\big),
\end{equation}
where $|\Omega|$ is the volume of $\Omega$.
\end{lemma}

\begin{lemma}[Lemma 2.3 in \cite{jost-Liu-Zhu-02}]\label{lem:struwe}
Suppose that $x_0\in M$ and $u\in C^2(M^T)$ where $T\leq\infty$. There exist constants $C>0$ and $R_0>0$ depending only on $M$, such that for any $r\in (0, R_0]$ and any function $\eta\in C^\infty_0(B^M_r(x_0))$ which depends only on the distance $|x-x_0|$ and is a non-increasing function of this distance, there holds
\begin{eqnarray}\label{e1}
\int_{M^T} |\nabla u|^4\eta d\mu dt&\leq& C\sup_{0\leq t\leq T}\int_{B^M_{r}(x_0)}|\nabla u|^2(x, t)d\mu\cdot\nonumber\\
&&(\int_{M^T} |\nabla^2 u|^2\eta d\mu dt+r^{-2}\int_{M^T}|\nabla u|^2\eta d\mu dt),
\end{eqnarray}
where $B^M_{r}(x_0)$ is the geodesic ball in $M$. Moreover, we have
\begin{eqnarray}\label{e2}
\int_{M^T} |\nabla u|^4 d\mu dt&\leq& C\sup_{(x_0, t)\in M^T}\int_{B^M_r(x_0)}|\nabla u|^2(x, t)d\mu\cdot\nonumber\\ &&(\int_{M^T}|\nabla^2 u|^2 d\mu dt+r^{-2}\int_{M^T}|\nabla u|^2 d\mu dt).
\end{eqnarray}
\end{lemma}
\begin{proof}
The idea is the same as that in \cite{St}, where Struwe used the density of step functions in $L^\infty$ space and  a covering argument. One can refer to Lemma 2.3 in \cite{jost-Liu-Zhu-02} for a detailed proof.
\end{proof}

Next we introduce several estimates which have been derived in \cite{HJLZ}. Although the a priori estimates in \cite{HJLZ} are derived in the space ${\mathcal{V}} (M_s^t;N\times {\mathbb{R}})$ (see the notation in \cite{HJLZ}), it is easy to check that the estimates also hold in the space $V(M_s^t)$.

\begin{lemma}[Lemma 2.1 in \cite{HJLZ}]\label{energy}
Suppose $(u,v)\in V(M^T)$ is a solution of (\ref{heateq}) and \eqref{boundary-data}, then the Lorentzian energy $E_g(u(t), v(t))$ is non-increasing on $[0,T)$ and for any $0\leq s\leq t<T$, there holds $$E_g(u(t), v(t))+\int_s^t\int_M|\partial_tu|^2dMdt\leq E_g(u(s), v(s)).$$
\end{lemma}

\begin{lemma}[Lemma 2.2 in\cite{HJLZ}, Corollary 2.3 in\cite{HJLZ}]\label{lem:Dirichlet-energy}
Suppose $(u,v)\in V(M^T)$ is a solution of (\ref{heateq}) and \eqref{boundary-data}, then for any $0\leq t<T$, there hold
\begin{eqnarray*}
\int_{M}|\nabla u|^2(\cdot,t)dM&\leq& \int_{M}|\nabla \phi_0|^2dM+\Lambda\int_{M}|\nabla \psi|^2dM,
\end{eqnarray*}
\begin{eqnarray*}
\int_{M}|\nabla v|^2(\cdot,t)dM&\leq& \frac{\Lambda}{\lambda}\int_{M}|\nabla \psi|^2dM
\end{eqnarray*}
and
\begin{eqnarray*}
\int_0^{T}\int_{M}|u_t|^2dMdt\leq E(\phi_0)+\Lambda E(\psi).
\end{eqnarray*}
\end{lemma}
\begin{proof}
By Lemma 2.2 in \cite{HJLZ}, we have $$\int_{M}\beta(u)|\nabla v|^2(\cdot,t)dM\leq \int_{M}\beta(u)|\nabla \psi|^2\ and \ \int_{M}|\nabla v|^2(\cdot,t)dM\leq \frac{\Lambda}{\lambda}\int_{M}|\nabla \psi|^2dM.$$ Then by Lemma \ref{energy}, we get
\begin{eqnarray*}
\frac{1}{2}\int_{M}|\nabla u|^2dx&\leq& E_g(u,v)+\frac{1}{2}\int_M \beta(u)|\nabla v|^2 dx\\
&\leq& E(\phi_0)+\frac{1}{2}\int_M \beta(u)|\nabla \psi|^2 dx\leq E(\phi_0)+\Lambda E(\psi)
\end{eqnarray*}
and
\begin{eqnarray*}
\int_0^t\int_{M}|u_t|^2dxdt&\leq& E(\phi_0)+\frac{1}{2}\int_{M}\beta(u)|\nabla v(\cdot, t)|^2dx\\
&\leq & \frac{1}{2}\int_{M}|\nabla \phi|^2dx+\frac{\Lambda}{2}\int_{M}|\nabla \psi|^2\leq E(\phi_0)+\Lambda E(\psi).
\end{eqnarray*}
\end{proof}

\begin{lemma}[Lemma 2.4 in\cite{HJLZ}]\label{lem:Lp-estimates-v}
Suppose $(u,v)\in V(M^T)$ is a solution of (\ref{heateq}) and \eqref{boundary-data}, then for any $p>1$ and $0\leq t<T$, we have
\begin{eqnarray*}
\int_{M}|\nabla v|^p(\cdot,t)dM\leq C\int_{M}|\nabla \psi|^pdM,
\end{eqnarray*}
where $C$ only depends on $p,M,\lambda,\Lambda$.
\end{lemma}

\begin{remark}\label{w1p}
Since we have the $W^{1,p}$ estimate for $v$, if the boundary data $\psi$ in \eqref{boundary-data} belongs to $C^{2,\alpha}$, we can replace the $L^4$-norm for $v$ in the definition of $V(M_s^t)$ by any $L^p$-norm with $p\geq 4$.
\end{remark}

\begin{lemma}[Lemma 2.5 in \cite{HJLZ}]\label{lem:two-balls}
Let $(u,v)\in V(M^T)$ be a solution to (\ref{heateq}) and \eqref{boundary-data}, then there exists a positive constant $R_0<1$ such that, for any $x_0\in M$, $0\leq r\leq R_0$ and $0<s\leq t<T$, there holds
\begin{eqnarray}\label{inequality:13}
E(u(t);B^M_r(x_0))\leq E(u(s);B^M_{2r}(x_0))+C_1\frac{t-s}{r^2}+C_2(t-s),
\end{eqnarray}
where $C_1$ and $C_2$ depend on $\lambda,\Lambda,M,N,E(\phi),\|\psi\|_{W^{1,4}(M)}$.
\end{lemma}

\begin{lemma}[Lemma 2.6 in \cite{HJLZ}]\label{lem:small-energy-regularity}
Let $(\phi,\psi)\in C^{2+\alpha}(\partial M,N\times \R)$, $z_0=(x_0,t_0)\in M\times (0, T]$ and denote $P_{r}^M(z_0):=B^M_r(x_0)\times [t_0-r^2,t_0]$. Assume that
$(u,v)$ is a solution of \eqref{heateq} and \eqref{boundary-data}, then there exists two positive constants $\epsilon_1=\epsilon_1(M,N,\phi,\psi)$ and
$C=C(\alpha,r,M,N,\|\phi\|_{C^{2+\alpha}( M)},\|\psi\|_{C^{2+\alpha}( M)})$, such that if
\[
\sup_{[t_0-4r^2,t_0]}E(u(t);B^M_r(x_0))\leq\epsilon_1,
\]
we have
\begin{equation}\label{inequality:01}
r\|\nabla v\|_{L^\infty(P_{r/2}^M(z_0))}+
r\|\nabla u\|_{L^\infty(P_{r/2}^M(z_0))}\leq C
\end{equation}
and for any $0<\beta<1$,
\begin{equation}\label{inequality:02}
\sup_{t_0-\frac{r^2}{4}\leq t\leq t_0}\|v(t)\|_{C^{2+\alpha}(B_{r/2}^M(x_0))}+
\|u\|_{C^{\beta,\beta/2}(P_{r/2}^M(z_0))}+\|\nabla u\|_{C^{\beta,\beta/2}(P_{r/2}^M(z_0))}\leq C(\beta).
\end{equation}

Moreover, if
\[
\sup_{x_0\in M}\sup_{[t_0-r^2,t_0]}E(u(t);B^M_r(x_0))\leq\epsilon_1
\]
then
\begin{equation}\label{inequality:03}
\sup_{t_0-\frac{r^2}{8}\leq t\leq t_0}\|v(t)\|_{C^{2+\alpha}(M)}+
\|u\|_{C^{2+\alpha,1+\alpha/2}(M\times [t_0-\frac{r^2}{8},t_0])}\leq C
\end{equation}
and
\begin{equation}\label{inequality:030}
\| v\|_{C^{\alpha,\alpha/2}(M\times [t_0-\frac{r^2}{8},t_0])}+\|\nabla v\|_{C^{\alpha,\alpha/2}(M\times [t_0-\frac{r^2}{8},t_0])}\leq C.
\end{equation}
\end{lemma}

\begin{lemma}\label{energyestimate}
Let $\phi_0\in W^{1,2}(M, N)$, $\phi\in C^{2+\alpha}(\partial M,N)$ and $\psi\in C^{2+\alpha}(\partial M, \R)$. Then there exist constants $\epsilon_2=\epsilon_2(M, N)>0$, $R_0=R_0(M)>0$ and $C=C(M, N)>0$, such that if $(u, v)\in V(M^T)$ is a solution of \eqref{heateq} and \eqref{boundary-data} and satisfies
\begin{eqnarray*}
\sup_{(x_0, t)\in M^T}\int_{B^M_r(x_0)}|\nabla u|^2 dM\leq \epsilon_2,\ \ \text{for all}\ \  r\in(0, R_0],
\end{eqnarray*}
there holds the estimate
\begin{eqnarray*}
&&E(u(T))+\int_{M^T} |\nabla^2 u|^2dMdt\\&&\leq C(1+\frac{T}{r^2})(E(\phi_0)+\Lambda E(\psi))+C\frac{T}{r^2}(\|\nabla\psi\|^4_{L^4(M)}+\|\phi\|^2_{C^2(M)}).
\end{eqnarray*}
\end{lemma}

\begin{proof}
Multiplying the first equation of \eqref{heateq} by $-\Delta u$ and integrating over $M^T$, we obtain that
\begin{eqnarray*}
&&E(u(T))-E(u(0))+\int_{M^T} |\Delta u|^2 dM dt\\
&=&-\int_{M^T}\partial_t u\Delta u dM dt+\int_{M^T} |\Delta u|^2 dM dt\\
&=&-\int_{M^T} (A(u)(\nabla u, \nabla u)-B^T(u)|\nabla v|^2)\Delta u dM dt\\
&\leq&\frac{1}{2}\int_{M^T}|\Delta u|^2 dM dt+C\int_{M^T}(|\nabla u|^4+|\nabla v|^4)dM dt.
\end{eqnarray*}
By Lemma \ref{lem:Lp-estimates-v} and Lemma \ref{lem:struwe}, we get
\begin{eqnarray}\label{inequality:05}
&&E(u(T))+\frac{1}{2}\int_{M^T} |\Delta u|^2 dM dt\notag\\
&&\leq E(u(0))+C\int_{M^T} |\nabla v|^4dM dt+C\int_{M^T} |\nabla u|^4dM dt\notag\\
&&\leq E(u(0))+C\int_{M^T}|\nabla \psi|^4dM dt+C\sup_{(x_0, t)\in M^T}\int_{B^M_r(x_0)}|\nabla u|^2 dM\notag\\
&&\quad\cdot(\int_{M^T}|\nabla^2 u|^2 dM dt+r^{-2}\int_{M^T}|\nabla u|^2 dM dt).
\end{eqnarray}
The standard elliptic estimates yield that
\[
\|\nabla^2 u\|_{L^2(M)}\leq C(M)(\|\Delta u\|_{L^2(M)}+\|\phi\|_{C^2(M)}),
\]
which implies
\begin{eqnarray}\label{inequality:04}
\int_{M^T}|\nabla^2 u|^2 dM dt\leq C(M)(\int_{M^T} |\Delta u|^2 dM dt+\|\phi\|^2_{C^2(M)}T).
\end{eqnarray}
Therefore,
\begin{eqnarray*}
&&E(u(T))+\frac{1}{2}\int_{M^T} |\Delta u|^2 dM dt\\
&&\leq C\epsilon_2\int_{M^T} |\Delta u|^2 dM dt+ E(u(0))+CT(\|\nabla\psi\|^4_{L^4(M)}+\|\phi\|^2_{C^2(M)})\\&&\quad+Cr^{-2}\int_{M^T}|\nabla u|^2 dM dt\\
&&\leq C\epsilon_2\int_{M^T} |\Delta u|^2 dM dt+ C(1+\frac{T}{r^2})(E(\phi_0)+\Lambda E(\psi))\\
&&\quad+C\frac{T}{r^2}(\|\nabla\psi\|^4_{L^4(M)}+\|\phi\|^2_{C^2(M)}),
\end{eqnarray*}
where the last inequality follows from Lemma \ref{lem:Dirichlet-energy}. Taking $\epsilon_2$ sufficiently small, we get that
\begin{eqnarray*}
\int_{M^T} |\Delta u|^2 dM dt\leq C(1+\frac{T}{r^2})(E(\phi_0)+\Lambda E(\psi))+C\frac{T}{r^2}(\|\nabla\psi\|^4_{L^4(M)}+\|\phi\|^2_{C^2(M)}).
\end{eqnarray*}
Then the desired conclusion follows immediately from \eqref{inequality:04}.
\end{proof}

Finally, we show the uniqueness result for (\ref{heateq}) and \eqref{boundary-data}.

\begin{theorem}
Let $\phi_0\in W^{1,2}(M,N)$, $\phi_0|_{\partial M}=\phi\in C^{2+\alpha}(\partial M,N)$ and
$\psi\in C^{2+\alpha}(\partial M)$. Suppose that $(u_i,v_i)\in V(M^T),i=1,2$ are weak solutions of \eqref{heateq} with the same boundary-initial data \eqref{boundary-data}, then $(u_1,v_1)\equiv(u_2,v_2)$ in $M^T$.
\end{theorem}
\begin{proof}
Suppose that $(u_i,v_i)\in V(M^T)$, $i=1,2$ are two weak solutions of \eqref{heateq} with the same boundary-initial data \eqref{boundary-data}. Let $U:=u_1-u_2$, $V:=v_1-v_2$ and denote $|\nabla U_{12}|=|\nabla u_1|+|\nabla u_2|$, $|\nabla V_{12}|=|\nabla v_1|+|\nabla v_2|$. By using the first equation in \eqref{heateq}, we have
\begin{eqnarray}\label{unique}
|\partial_t U-\Delta U|&\leq&|A(u_1)(\nabla u_1,\nabla u_1)-A(u_2)(\nabla u_2,\nabla u_2)|\notag\\&&+|B^\top(u_1)|\nabla v_1|^2-B^\top(u_2)|\nabla v_2|^2|\nonumber\\
&\leq&C(|\nabla U||\nabla U_{12}|+|U||\nabla U_{12}|^2)+C(|U||\nabla V_{12}|^2+|\nabla V||\nabla V_{12}|).
\end{eqnarray}
Multiplying \eqref{unique} by $U$ and integrating over $M^{t_0}$, we get
\begin{align*}
&\frac{1}{2}\int_{M} |U|^2(x,{t_0}) dM+\int_{M^{t_0}} |\nabla U|^2dM dt \nonumber\\
&=\frac{1}{2}\int_{M^{t_0}}\partial_t |U|^2dMdt-\int_{M^{t_0}}\Delta U \cdot U dMdt\nonumber\\
&\leq C\int_{M^{t_0}}|U||\nabla U||\nabla U_{12}|dM dt
+C\int_{M^{t_0}}|U|^2|\nabla U_{12}|^2dM dt\nonumber\\
&\quad+C\int_{M^{t_0}}|U|^2|\nabla V_{12}|^2dMdt+C\int_{M^{t_0}}|U||\nabla V||\nabla V_{12}|dM dt\nonumber\\
&\leq C(\int_{M^{t_0}}|U|^4dM dt)^{1/4}(\int_{M^{t_0}}|\nabla U|^2dM dt)^{1/2}
(\int_{M^{t_0}}|\nabla U_{12}|^4dM dt)^{1/4}\nonumber\\
&\quad+C(\int_{M^{t_0}}|U|^4dM dt)^{1/2}
[(\int_{M^{t_0}}|\nabla U_{12}|^4dM dt)^{1/2}+(\int_{M^{t_0}}|\nabla V_{12}|^4dM dt)^{1/2}]\nonumber\\
&\quad+C(\int_{M^{t_0}}|U|^4dM dt)^{1/4}(\int_{M^{t_0}}|\nabla V|^2dM dt)^{1/2}
(\int_{M^{t_0}}|\nabla V_{12}|^4dM dt)^{1/4}\nonumber\\
&\leq C\epsilon({t_0})(\int_{M^{t_0}}|U|^4dM dt)^{1/4}(\int_{M^{t_0}}|\nabla U|^2dM dt)^{1/2}+C\epsilon({t_0})(\int_{M^{t_0}}|U|^4dM dt)^{1/2}\nonumber\\
&\quad+C\epsilon({t_0})(\int_{M^{t_0}}|U|^4dM dt)^{1/4}(\int_{M^{t_0}}|\nabla V|^2dM dt)^{1/2}\nonumber\\
&\leq C\epsilon({t_0})(\int_{M^{t_0}}|U|^4dM dt)^{1/2}
+\frac{1}{2}\int_{M^{t_0}}|\nabla U|^2dM dt
+\frac{1}{2}\int_{M^{t_0}}|\nabla V|^2dM dt,
\end{align*}
where ${t_0}\in (0,T]$ and $\epsilon({t_0})\rightarrow 0$ as ${t_0}\rightarrow 0$. The second equation of \eqref{heateq} gives
\begin{equation}\label{div}
-div\left(\beta(u_1)\nabla v_1-\beta(u_2)\nabla v_2\right)=0.
\end{equation}
Multiplying \eqref{div} by $V$ and integrating by parts on $M^{t_0}$, we get
\begin{eqnarray}\label{vestimate1}
\int_{M^{t_0}}\beta(u_2)|\nabla V|^2dx dt
=\int_{M^{t_0}}(\beta(u_2)-\beta(u_1))\nabla v_1\cdot \nabla V dx dt.
\end{eqnarray}
Then we get from \eqref{vestimate1} that
\begin{eqnarray}\label{vestimate}
&&\int_{M^{t_0}}|\nabla V|^2dx dt\nonumber\\
&\leq& C\int_{M^{t_0}}|U||\nabla V_{12}||\nabla V| dx dt\nonumber\\
&\leq&C\left(\int_{M^{t_0}}|U|^4 dx dt\right)^{1/4}
\left(\int_{M^{t_0}}|\nabla V|^2 dx dt\right)^{1/2}
\left(\int_{M^{t_0}}|\nabla V_{12}|^4 dx dt\right)^{1/4}\nonumber\\
&\leq&C\epsilon({t_0})\left(\int_{M^{t_0}}|U|^4 dx dt\right)^{1/4}
\left(\int_{M^{t_0}}|\nabla V|^2 dx dt\right)^{1/2}\nonumber\\
&\leq&C\epsilon(t)\left(\int_{M^{t_0}}|U|^4 dx dt\right)^{1/2}
+\frac{1}{2}\int_{M^{t_0}}|\nabla V|^2 dx dt
.
\end{eqnarray}
Therefore, we get
\begin{eqnarray}\label{inequality:09}
\frac{1}{2}\int_{M} |U|^2(x,t_0) dx+\frac{1}{2}\int_{M^{t_0}} |\nabla U|^2dx dt\leq C\epsilon({t_0})\left(\int_{M^{t_0}}|U|^4dx dt\right)^{1/2}
.
\end{eqnarray}
By Lemma \ref{lem:01}, we know that
\begin{eqnarray}\label{inequality:08}
\int_{M^{t_0}} |U|^4dx dt
&\leq&C\int_0^{t_0} \int_M |U|^2dx\left(\int_M|\nabla U|^2dx+\int_M|U|^2dx\right)dt\nonumber\\
&\leq&C\sup_{0\leq s\leq {t_0}}\int_M |U|^2(x,s)dx\left(\int_{M^{t_0}}|\nabla U|^2dx dt+\int_{M^{t_0}}|U|^2dx dt\right)\nonumber\\
&\leq&C\left(\sup_{0\leq s\leq {t_0}}\int_M |U|^2(x,s)dx+\int_{M^{t_0}}|\nabla U|^2dx dt\right)^2.
\end{eqnarray}
Substituting \eqref{inequality:08} into \eqref{inequality:09}, we get
\begin{eqnarray}
&&\frac{1}{2}\int_{M} |U|^2(x,t_0) dx+\frac{1}{2}\int_{M^{t_0}} |\nabla U|^2dx dt\nonumber\\
&\leq&C\epsilon({t_0})\left(\sup_{0\leq s\leq t_0}\int_M |U|^2(x,s)dx+\int_{M^{t_0}}|\nabla U|^2dx dt\right).
\end{eqnarray}
Without loss of generality, we may assume
$$
\int_M|U|^2(\cdot, t_0)dx=\sup_{0\leq s\leq {t_0}}\int_M |U|^2(\cdot,s)dx.
$$
Since $\epsilon(t_0)\to 0$ as $t_0\to 0$, then there exists $s_0\in (0,T]$ such that
$$
\int_{M} |U|^2(x,s_0) dx+\int_{M^{s_0}} |\nabla U|^2dx dt=0,
$$
which implies that $U\equiv 0$ in $M^{s_0}$ and consequently by \eqref{vestimate}, we get that $V\equiv 0$ in $M^{s_0}$. We can repeat this process at $t=s_0$ and finally get the theorem proved by iteration.
\end{proof}

\section{Proof of Theorem \ref{global}}

%
%

Now we can begin to prove Theorem \ref{global}.

\begin{proof}[\textbf{Proof of Theorem \ref{global}}]
\textbf{Step 1:}
Let $g\in C^{2+\alpha}(M)$ be the unique solution to the equation
\begin{equation*}
\begin{cases}
-\Delta g=0\ &\text{in}\ M;\\
g=\phi\ &\text{on}\ \partial M.
\end{cases}
\end{equation*}
Since $\phi_0-g\in W_0^{1,2}(M)$, we can choose $\varphi_{0m}\in C_0^\infty(M)$ such that $\varphi_{0m}\rightarrow \phi_0-g$ in $W^{1,2}(M)$. Take $\phi_{0m}=\varphi_{0m}+g\in C^{2+\alpha}(M)$. We have that $\phi_{0m}|_{\partial M}=\phi$ and
\begin{eqnarray*}
\phi_{0m}&\rightarrow& \phi_0\ \text{in} \ W^{1,2}(M);
\end{eqnarray*}

\textbf{Step 2:} Short-time existence. By Theorem \ref{thm:shortime-existence}, we know that there exist $T_m>0$ and $u_m\in  C_{loc}^{2+\alpha,1+\alpha/2}(M\times [0,T_m)), $ $v_m,\nabla v_m\in  C_{loc}^{\alpha,\alpha/2}(M\times [0,T_m))$ which solve  \eqref{heateq} on $M\times [0,T_m)$ with the boundary-initial date $\phi_{0m}, \phi, \psi$.

Since $\phi_{0m}\rightarrow\phi_0$ strongly in $W^{1,2}(M)$, there exists some $R>0$ such that for all $x\in M$,
\begin{equation}\label{equat:02}
E(\phi_{0m}; B^M_{2R}(x))<\frac{\overline{\epsilon}}{4},\end{equation} where $\overline{\epsilon}:=\min\{\epsilon_1,\epsilon_2\}$ and $\epsilon_1$, $\epsilon_2$ are the constants in Lemma \ref{lem:small-energy-regularity} and Lemma \ref{energyestimate} respectively. By Lemma \ref{lem:two-balls}, if $T=O(R^2\overline{\epsilon})$,
$$\sup_{(x,t)\in M^T}E(u_{m}(\cdot, t); B^M_{R}(x))<\overline{\epsilon}.$$
By Lemma \ref{lem:two-balls}, Lemma \ref{lem:small-energy-regularity} and Theorem \ref{thm:shortime-existence}, we may assume $T_m\geq T=O(R^2\overline{\epsilon})$. Then, Lemma \ref{energy}, Lemma \ref{lem:Dirichlet-energy}, Lemma \ref{lem:Lp-estimates-v} and Lemma \ref{energyestimate} yield
$$
\|(u_m, v_m)\|_{V(M^T)}\leq C.
$$
According to the weak compactness, there exists a subsequence of $\{(u_m, v_m)\}$ (denoted by itself) and $(u, v)\in V(M^T)$ such that
$$
(u_m, v_m)\rightharpoonup (u,v)\quad \text{weakly in}\ V(M^T).
$$
It is easy to check that $(u,v)$ is a weak solution of \eqref{heateq} and \eqref{boundary-data} in the sense of distribution.

By Lemma \ref{lem:small-energy-regularity}, we have
$$\|u_m\|_{C^{2+\alpha,1+\frac{\alpha}{2}}(M\times [\delta,T])}+\|v_m\|_{C^{\alpha,\frac{\alpha}{2}}(M\times [\delta,T])}+\|\nabla u_m\|_{C^{\alpha,\frac{\alpha}{2}}(M\times [\delta,T])}+\sup_{\delta\leq t\leq T}\|v_m\|_{C^{2+\alpha}(M)}\leq C,$$ where $C=C(\alpha,R, \delta, T, M,N)$ is a positive constant, which implies that
$$
u\in C_{loc}^{2+\alpha,1+\alpha/2}(M\times (0,T])
$$
and
$$
v,\nabla v\in C_{loc}^{\alpha,\alpha/2}(M\times (0,T]),\quad
v\in L_{loc}^{\infty}((0,T];C^{2+\alpha}(M)).
$$
Thus $(u,v)$ is a classical solution of \eqref{heateq} and \eqref{boundary-data} defined on $M\times (0,T]$. By Theorem \ref{thm:shortime-existence} and using $(u(T),v(T))$ as the new boundary-initial data of \eqref{heateq}, the solution can be extended to a larger time interval. This argument remain feasible until we reach the first singular time $T_1$ when the condition
\begin{eqnarray}
\limsup_{x\in M,t\to T}E(u;B^M_r(x))>\overline{\epsilon}\mbox{ for any } r>0,
\end{eqnarray}
is satisfied. Let $\{x^l,T_1\}_{l=1}^{l_1}$ be any finite set satisfying
\[
\limsup_{t\to T_1}E(u;B^M_r(x^l))>\overline{\epsilon}\mbox{ for any } r>0,\ 1\leq l\leq l_1.
\]
Choosing $r>0$ such that $B^M_{2r}(x^l)$, $1\leq l\leq l_1$, are mutually disjoint. By Lemma \ref{lem:Dirichlet-energy}, we have
\begin{align*}
l_1\overline{\epsilon}&\leq \sum_{1\leq l\leq l_1}\limsup_{t\nearrow T_1}E(u(t);B^M_r(x^l))\\&\leq \sum_{1\leq l\leq l_1}E(u(s);B^M_{2r}(x^l))+\frac{l_1\overline{\epsilon}}{2}\\
&\leq E(u(s))+\frac{l_1\overline{\epsilon}}{2}\leq E(\phi_0)+\Lambda E(\psi)+\frac{l_1\overline{\epsilon}}{2},
\end{align*} where $0<s<T_1$ is sufficiently closed to $T_1$, such that $C\frac{T_1-s}{r^2}\leq\frac{\overline{\epsilon}}{2}$.
Thus,
\begin{equation}\label{inequality:11}
l_1\leq 2\frac{E(\phi_0)+\Lambda E(\psi)}{\overline{\epsilon}}.
\end{equation}
From Lemma \ref{lem:small-energy-regularity}, it is easy to see that
$u\in C_{loc}^{2+\alpha,1+\alpha/2}(M\times (0,T_1))$ and $\ u,\nabla u\in C_{loc}^{\alpha,\alpha/2}(M\times (0,T_1])$ outside singular points $\{(x^l,T_1)\}_{l=1}^{l_1}$.
Moreover, $\ v,\nabla v\in C_{loc}^{\alpha,\alpha/2}(M\times (0,T_1))$  and $\ v,\nabla v\in C_{loc}^{\alpha,\alpha/2}(M\times (0,T_1])$ outside singular points  $\{(x^l,T_1)\}_{l=1}^{l_1}$.

\textbf{Step 3:} Global existence of weak solution. By the proof of Theorem 1.4 in \cite{HJLZ-02}, there is a unique weak limit $(u(T_1),v(T_1))\in W^{1,2}(M,N\times\R)$ with boundary data $(u(T_1),v(T_1))|_{\partial M}=(\phi, \psi)$, such that $$
(u(t),v(t))\rightharpoonup (u(T_1),v(T_1))\ \text{weakly in}\ W^{1,2}(M)
$$
as $t\to T_1$. Moreover, $$\lim_{t\to T_1}\int_M\beta(u(t))|\nabla v(t)|^2dM=\int_M\beta(u(T_1))|\nabla v(T_1)|^2dM.$$
Noting that,
\begin{eqnarray}
E(u(T_1))&=&\lim_{r\rightarrow 0}E(u(T_1); M\setminus \cup_{l=1}^{l_1}B_r^M(x^l))\nonumber\\
&\leq&\lim_{r\rightarrow 0}\liminf_{t\rightarrow T_1}E(u(t); M\setminus \cup_{l=1}^{l_1}B_r^M(x^l))\nonumber\\
&=&\liminf_{t\rightarrow T_1}E(u(t))-\lim_{r\rightarrow 0}\limsup_{t\rightarrow T_1}\sum_{l=1}^{l_1}E(u(t);B_r^M(x^l))\nonumber\\
&\leq&\liminf_{t\rightarrow T_1}E(u(t))-\overline{\epsilon},
\end{eqnarray}
we have
\begin{equation}\label{inequality:10}
E_g(u(T_1),v(T_1))\leq \liminf_{t\rightarrow T_1}E_g(u(t),v(t))-\overline{\epsilon}.
\end{equation}

According to \textbf{Step 2}, we can get a solution of the system \eqref{heateq} on $[T_1,T_2)$ for some $T_2>T_1$ by viewing $u(T_1),\phi,\psi$ as the new boundary-initial data. Piecing together the solutions at $T_1$, we get a weak solution of \eqref{heateq} on $M\times [0,T_2)$. Iterating this process, we obtain a global weak solution defined on $M\times [0,\infty)$. Let $\{T_k\}_{k=1}^K$ be the singular times and $\{(x^l,T_k)\}_{l=1}^{l_k}$ be the singular points at time $t=T_k$. By \eqref{inequality:10} and Lemma \ref{energy}, we have
\begin{align*}
E_g(u(T_K),v(T_K))&\leq \liminf_{t\rightarrow T_K}E_g(u(t),v(t))-\overline{\epsilon}\\
&\leq \liminf_{t\rightarrow T_1}E_g(u(t),v(t))-K\overline{\epsilon}\\
&\leq \liminf_{t\rightarrow T_1}E(u(t))-K\overline{\epsilon}\leq E(\phi_0)+\Lambda E(\psi)-K\overline{\epsilon}.
\end{align*}
Thus,
\[
K\leq \frac{E(\phi_0)+\Lambda E(\psi)}{\overline{\epsilon}}.
\]
Combing this with \eqref{inequality:11}, we get
\[
\sum_{k=1}^Kl_k\leq 2(\frac{E(\phi_0)+\Lambda E(\psi)}{\overline{\epsilon}})^2.
\]

\textbf{Step 4:} Convergence. By Lemma \ref{energy} and \ref{lem:Dirichlet-energy}, we know that
$$
\int_0^\infty |\partial_t u|^2dM dt
+\sup_{0\leq t<\infty}E(u(\cdot,t))
+\sup_{0\leq t<\infty}E(v(\cdot,t))
\leq C<\infty.
$$
Then, we have that there exists a time sequence
$t_i\rightarrow \infty$, such that
$$
\|\partial_t u(\cdot, t_i)\|_{L^2}\rightarrow 0.
$$
Combing this with the small energy regularity theory for approximate Lorentzian harmonic maps (see Lemma 2.1 and Lemma 2.2 in \cite{HJLZ-02}), we know that there exist a positive constant $\overline{\epsilon}'$ depending only on $\lambda,\Lambda,N\times\R$, a finite points set
\begin{equation}\label{def:01}
\mathcal{S}_\infty:=\{x\in M| \liminf_{i\to\infty} E(u(\cdot,t_i);B^M_r(x))>\overline{\epsilon}'\ for\ any\ r>0\}\end{equation} and a Lorentzian harmonic map $(u_\infty,v_\infty)\in W^{1,2}(M,N\times\R)\cap W^{2,2}_{loc}(M\setminus \mathcal{S}_\infty,N\times\R)$ with boundary data $(u_\infty,v_\infty)|_{\partial M}=(\phi,\psi)$ such that, up to a subsequence,
$(u,v)(\cdot,t_i)$ converges weakly in $W^{1,2}(M)$ and strongly in $W^{2,2}_{loc}(M\setminus \mathcal{S}_\infty)$ to $(u_\infty,v_\infty)$.  By removable singularity Theorem 2.8 in \cite{HJLZ-02}, we have $(u_\infty,v_\infty)\in W^{2,2}(M,N\times\R)$ which implies $(u_\infty,v_\infty)\in C^{2+\alpha}(M,N\times\R)$ by the standard elliptic theory of Laplace operator.
\end{proof}

\end{document}